\documentclass[12pt]{amsart}
\usepackage{etex}

\usepackage{amsthm,amssymb,enumitem,mathtools}

\usepackage{tikz}

\usepackage{scalerel}

\usetikzlibrary{arrows,shapes,automata,backgrounds,decorations,petri,positioning}
\usetikzlibrary[calc,intersections,through,backgrounds,arrows,decorations.pathmorphing]

\tikzstyle{edge} = [fill,opacity=.5,fill opacity=.5,line cap=round, line join=round, line width=50pt]

\pgfdeclarelayer{background}
\pgfsetlayers{background,main}

\usepackage[margin=1in,letterpaper,portrait]{geometry}

\usepackage[latin1]{inputenc}
\usepackage{subfigure}
\usepackage{color}
\usepackage{amsmath}
\usepackage{amstext}
\usepackage{amssymb}
\usepackage{amsfonts}
\usepackage{graphicx}
\usepackage{young}
\usepackage{multicol}
\usepackage{mathrsfs}
\usepackage[all]{xy}
\usepackage{tikz}
\usepackage{mathtools}
\usepackage{tabularx}
\usepackage{array}
\usepackage{commath}
\usepackage{ytableau}

\usepackage{scrextend}

\usepackage{hyperref}

\theoremstyle{plain}
\theoremstyle{definition}
\newtheorem{theorem}{Theorem}[section]

 \newtheorem{lemma}[theorem]{Lemma}
 
 \newtheorem{definition}[theorem]{Definition}

 \newtheorem{example}[theorem]{Example}
 
 \newtheorem{corollary}[theorem]{Corollary}

\DeclareMathAlphabet{\mathpzc}{OT1}{pzc}{m}{it}

\newcommand{\mc}[1]{\mathcal{#1}}
\newcommand{\symm}{\mathfrak{S}}

\DeclareMathOperator{\Pin}{Pin}

\begin{document}

\title{Admissible pinnacle orderings}

\author{Irena Rusu}
\address{LS2N, UMR 6004, Universit\'e de Nantes, France}
\email{Irena.Rusu@univ-nantes.fr}

\author[Bridget Eileen Tenner]{Bridget Eileen Tenner$^*$}
\address{Department of Mathematical Sciences, DePaul University, Chicago, IL, USA}
\email{bridget@math.depaul.edu}
\thanks{$^*$Research partially supported by Simons Foundation Collaboration Grant for Mathematicians 277603 and by a University Research Council Competitive Research Leave from DePaul University.}



\begin{abstract}
A {\em pinnacle} of a permutation is a value that is larger than its immediate neighbors when written in one-line notation. In this paper, we build on previous work that characterized admissible pinnacle sets of permutations. For these sets, there can be specific orderings of the pinnacles that are not admissible, meaning that they are not realized by any permutation. Here we characterize admissible orderings, using the relationship between a pinnacle $x$ and its rank in the pinnacle set to bound the number of times that the pinnacles less than or equal to $x$ can be interrupted by larger values.
\end{abstract}

\maketitle

\section{Introduction}\label{sec:intro}

In previous work with Davis, Nelson, and Petersen, the second author defined and studied pinnacle sets of permutations \cite{davis nelson petersen tenner}. This was motivated by the related analyses of peak sets in the literature, including work by Billey, Burdzy, and Sagan \cite{billey burdzy sagan}. The work of \cite{davis nelson petersen tenner} included a variety of related enumerative results about these objects, and---most relevant to the present work---a characterization of which sets can occur as pinnacle sets of permutations. Such sets were called \emph{admissible} pinnacle sets. Recent works by the first author \cite{rusu} and by Diaz-Lopez, Harris, Huang, Insko, and Nilsen \cite{diaz-lopez harris huang insko nilsen} devise procedures for generating all permutations with a given pinnacle set.

It was recently noted by the first author that an admissible pinnacle set may not actually be admissible under all orderings \cite{rusu}. For example, \{3,5,7\} is an admissible pinnacle set, as demonstrated by the permutation $4523176 \in \symm_7$ and depicted in Figure~\ref{fig:4523176}, but there is no permutation with these pinnacles for which they appear in the order $(3,7,5)$.
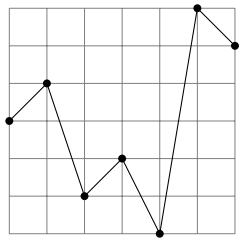
\begin{figure}[htbp]
\centering
\begin{tikzpicture}[scale=.5]
\draw[step=1.0,gray,very thin] (1,1) grid (7,7);
\draw (1,4) -- (2,5) -- (3,2) -- (4,3) -- (5,1) -- (6,7) -- (7,6);
\foreach \x in {(1,4), (2,5), (3,2), (4,3), (5,1), (6,7), (7,6)} {\fill \x circle (3pt);}
\end{tikzpicture}
\caption{The permutation $4523176 \in \symm_7$. Its pinnacle set is $\{3,5,7\}$ and the pinnacles appear in the order $(5,3,7)$.}
\label{fig:4523176}
\end{figure}
Question 3 of \cite{rusu} was to characterize which orderings of the elements in an admissible pinnacle set can be realized.

Here we answer that question, determining which orderings of an admissible pinnacle set are, themselves, admissible. This is, as it turns out, not so much about permutations of the pinnacles as it is about subsets of the pinnacles that must be uninterrupted when they appear in a permutation. We foreshadow that result with the following example.

\begin{example}\label{ex:multilevel bundling}
The set $S = \{3, 5, 8, 9, 13, 14\}$ is an admissible pinnacle set. In any permutation $w$ with pinnacle set $S$, the pinnacles can appear in any order so long as
\begin{itemize}
\item[$\bullet$] the pinnacles $\{3,5\}$ appear consecutively, and
\item[$\bullet$] the pinnacles $\{3,5,8,9\}$ appear consecutively.
\end{itemize}
Thus, there are $2!3!3! = 72$ admissible orderings of the pinnacle set $S$: permutations of the elements of each of the sets
$$\{3,5\}, \ \big\{ \{3,5\}, 8, 9\big\}, \text{ and } \ \Big\{ \big\{ \{3,5\}, 8, 9\big\}, 13, 14\Big\}.$$
This is one tenth of the total orderings of the $6$ elements of $S$.
\end{example}

In Section~\ref{sec:background}, below, we introduce the terminology, notation, and key background for the work in this paper. Section~\ref{sec:main result} contains the main result and we conclude with directions for further research in Section~\ref{sec:follow-up}.

\section{Background}\label{sec:background}

In this paper we consider permutations of $[n]$ as words $w = w(1)w(2)\cdots w(n)$. Let $\symm_n$ be the set of all such permutations.

\begin{definition}
A \emph{peak} of $w$ is an index $i \in [2,n-1]$ such that $w(i-1) < w(i) > w(i+1)$, and a valley is an index $i$ such that $w(i-1) > w(i) < w(i+1)$. Focusing on values and not positions, a \emph{pinnacle} of $w$ is a value $w(i)$ at which $i$ is a peak, and a \emph{vale} of $w$ is a value $w(i)$ at which $i$ is a valley. The pinnacles of a permutation $w$ are denoted $\Pin(w)$.
\end{definition}

\begin{example}
Consider $13287564 \in \symm_8$. The peaks of this permutation are $\{2,4,7\}$, the valleys are $\{3,6\}$, the pinnacles are $\{3,6,8\}$, and the vales are $\{2,5\}$. The pinnacles appear in the order $(3,8,6)$.
\end{example}

When one \emph{graphs} a permutation by plotting and connecting the points $\{(i,w(i))\}$, as in Figure~\ref{fig:4523176}, the resulting ``landscape'' gives geographic motivation to the words ``peak,'' ``valley,'' ``pinnacle,'' and ``vale.''

Not every set can be a pinnacle set. For example $\{2\}$ is not a pinnacle set because there are not two different positive integers, both less than $2$, to serve as neighbors to $2$ in a permutation.

\begin{definition}
A set $S$ is an \emph{admissible} pinnacle set if there exists a permutation whose pinnacle set is $S$.
\end{definition}

The characterizing result of \cite{davis nelson petersen tenner} was the following theorem.

\begin{theorem}[{\cite[Theorem 1.5]{davis nelson petersen tenner}}]\label{thm:admissible sets}
Let $S$ be a set of integers with $\max S = m$. Then $S$ is an admissible pinnacle set if and only if both
\begin{enumerate}
\item[1.] $S \setminus \{m\}$ is an admissible pinnacle set, and
\item[2.] $m > 2|S|$.
\end{enumerate}
\end{theorem}

As discussed in that paper, an admissible pinnacle set with maximum element $m$ can be studied in $\symm_m$ without losing any sense of generality. This is because non-pinnacles that are larger than this $m$ would have to appear in a decreasing prefix or increasing suffix, adding no information or restriction about the pinnacles that may appear.

Note that Theorem~\ref{thm:admissible sets} says nothing about the order in which the pinnacles can (or must) appear in a demonstrative permutation.

\begin{definition}
Let $S$ be an admissible pinnacle set. An \emph{admissible ordering} of $S$ is an ordering of the elements of $S$ for which there exists a permutation whose pinnacle set is $S$ and whose pinnacles appear in the given order.
\end{definition}

\begin{example}\
\begin{enumerate}\label{ex:357}
\item[(a)] The set $S = \{3,5,7\}$ is an admissible pinnacle set. The admissible orderings of $S$ are $(3,5,7)$, $(5,3,7)$, $(7,3,5)$, and $(7,5,3)$. Demonstrative permutations for each of these admissible orderings are, respectively, $1325476$, $4513276$, $6713254$, and $6745132$.
\item[(b)] Neither $(3,7,5)$ nor $(5,7,3)$ is an admissible ordering. In both cases the ordering is impossible because there would be nowhere to place $6$ in such a permutation.
\end{enumerate}
\end{example}

One might read Example~\ref{ex:357}(b) to suggest a whiff of permutation patterns to be the issue of admissible orderings. Something like, perhaps: a large pinnacle cannot be sandwiched between smaller pinnacles. However, the actual result has a notably different flavor. Indeed, the main result applied to the pinnacle set $\{3,5,7\}$---in fact, 
to any admissible pinnacle set of the form $\{3,5,x\}$---would say only that the pinnacles $3$ and $5$ must appear consecutively.

\section{Main result}\label{sec:main result}

Throughout this section, let $S$ be an admissible pinnacle set. To ease the discussion, we set the following notation, always assuming that $x$ is a positive integer:
\begin{align*}
\text{non-pinnacles:} \ \ \overline{S} &:= [1,\max S] \setminus S,\\
\text{small pinnacles:} \ \ S_x &:= [1,x] \cap S, \text{ and}\\
\text{small non-pinnacles:} \ \ \overline{S}_x &:= [1,x] \cap \overline{S} = [1,x] \setminus S = [1,x] \setminus S_x.
\end{align*}
Certainly $|S_x| + |\overline{S}_x| = x$.

As suggested by Example~\ref{ex:357}, what might prevent an ordering of $S$ to be admissible is not an element of $S$ itself, but rather the elements of $\overline{S}$. 

\begin{definition}
Let $S$ be a set, and $T \subseteq S$. Fix an ordering $\mc{A}$ of the elements of $S$. Abuse notation for a moment and consider $\mc{A}$ not as a sequence but as a word
$$\mc{A} = a_0 t_1 a_1 t_2 \cdots a_{k - 1} t_k a_k,$$
where each $t_i$ is a word consisting of elements of $T$, each $a_i$ is a word consisting of elements of $S \setminus T$, and only $a_0$ and $a_k$ may be empty. The set $T$ is \emph{interrupted} $k - 1$ times in $\mc{A}$. If $k = 1$, we say that $T$ is \emph{uninterrupted} in $\mc{A}$.
\end{definition}

We demonstrate this interruption with two examples.

\begin{example}\label{ex: 4,6,8,10,11}\
\begin{enumerate}
\item[(a)] For any $S$ and $\mc{A}$, the set $S$ itself is uninterrupted in $\mc{A}$.
\item[(b)] Let $S = \{4,6,8,10,11\}$. The set $\{4,6,8\} \subset S$ is interrupted $1$ time in the ordering $(10,6,4,11,8)$, and $2$ times in the ordering $(6,10,4,11,8)$.
\end{enumerate}
\end{example}

The language of interruption succinctly describes a phenomenon of admissible pinnacle orderings.

\begin{lemma}\label{lem:few valleys means uninterrupted}
For any $x \in S$, if
$$\big|S_x \big| = \big|\overline{S}_x\big| - 1,$$
then $S_x$ is uninterrupted in the admissible orderings of $S$.
\end{lemma}

\begin{proof}
Set $h := |S_x|$. These smallest $h$ pinnacles must be interspersed with at least $h+1$ elements of $\overline{S}_x$. But $|\overline{S}_x| = h+1$, so there are only $h+1$ elements of $\overline{S}$ available for this purpose. Thus, in $w$, these $2h+1$ values must appear consecutively, in some order that yields the required pinnacles. Thus the pinnacles in $S_x$ are uninterrupted in the admissible orderings of $S$.
\end{proof}

Lemma~\ref{lem:few valleys means uninterrupted} suggests a rephrasing of Theorem~\ref{thm:admissible sets} that will give key insight to the question of admissible orderings. One appeal of this alternative characterization is that it is not recursive.

\begin{theorem}\label{thm:admissible sets rephrased}
A set $S$ of positive integers is an admissible pinnacle set if and only if $|S_x| < x/2$ for all $x \in S$.
\end{theorem}

\begin{proof}
We prove Theorem~\ref{thm:admissible sets rephrased} by induction on $|S|$. If $S$ is an admissible $1$-element pinnacle set, then certainly $x \in S$ is at least $3$, and $|S_x| = 1 < 3/2 \le x/2$. Similarly, if $S = \{x\}$ and $1 < x/2$, then $x \ge 3$ and $S$ is clearly admissible: the permutation $1x234\cdots (x-1)$ has pinnacle set $\{x\}$.

Now suppose that the result holds for all sets with at most $n$ elements, and suppose that $|S| = n+1$. Let $m = \max S$. Suppose, first, that $S$ is admissible. By Theorem~\ref{thm:admissible sets}, $|S| = |S_m| < m/2$ and $S \setminus \{m\}$ must be admissible. Therefore $|S_x| < x/2$ for all $x \in S\setminus \{m\}$, by the inductive hypothesis, completing the proof of this direction. On the other hand, if $|S_x| < x/2$ for all $x \in S$, then $|S_m| = |S| < m/2$, and $S \setminus \{m\}$ is admissible by the inductive hypothesis. It follows from Theorem~\ref{thm:admissible sets}, then, that $S$ is admissible.
\end{proof}

As observed above, the inequality ``$|S_x| < x/2$'' in Theorem~\ref{thm:admissible sets rephrased} could be rephrased as
$$|S_x| < |\overline{S}_x|.$$

One corollary of Theorem~\ref{thm:admissible sets rephrased} is that admissibility is maintained in ``down-sets.''

\begin{corollary}\label{cor:down-sets}
If $S$ is an admissible pinnacle set, then $S_x$ is an admissible pinnacle set for all $x$.
\end{corollary}

We are now able to state the main result, answering the question in \cite{rusu}. That result will be stated in terms of the statistic
$$k_x := |\overline{S}_x| - |S_x| - 1,$$
for  $x \in S$.

\begin{theorem}\label{thm:main}\label{thm:admissible orders}
Let $S$ be an admissible pinnacle set. An ordering $\mc{A}$ of $S$ is admissible if and only if, for each $x \in S$, the set $S_x$ is interrupted at most $k_x$ times in $\mc{A}$.
\end{theorem}

\begin{proof}

First suppose that $\mc{A}$ is admissible. The vales that are interspersed among the pinnacles $y \in S_x$ must be less than $x$, and hence are elements of $\overline{S}_x$. There are only $|S_x| + k_x + 1$ such values available. For $S_x$ to be interrupted more than $k_x$ times would require at least $|S_x| + k_x + 2$ valleys/vales: one vale to the left of each element of $S_x$, one vale at the start of each interruption, and one value to the right of the rightmost element in $S_x$. There are too few vales available, so $S_x$ must be interrupted at most $k_x$ times.

Now suppose that we have an ordering $\mc{A}$ of $S$ in which, for each $x \in S$, the set $S_x$ is interrupted at most $k_x$ times. We will show that $\mc{A}$ is admissible by constructing a permutation $w$ whose pinnacles appear in this order. The construction will use the smallest $|S| + 1$ elements of $\overline{S}$ as vales, which is sufficient to show that $\mc{A}$ is admissible because the remaining elements of $\overline{S}$ can be added as a decreasing prefix and/or increasing suffix.

Let $S = \{x_1, \ldots, x_p\}$ with $x_1 < x_2 < \cdots < x_p$. Use the two smallest elements in $\overline{S}_{x_1}$, which are necessarily $1$ and $2$, as vales incident to $x_1$. Now assume, inductively, that all pinnacles in $S_{x_i}$ have their incident vales filled with the smallest elements in $\overline{S}$. The set $S_{x_i}$ is interrupted $j_{x_i} \le k_{x_i}$ times, so we have used $|S_{x_i}| + j_{x_i} + 1 \le |\overline{S}_{x_i}|$ elements of $\overline{S}_{x_i}$ to do this.

We must now fill in the vales incident to the pinnacles in $S_{x_{i+1}} \setminus S_{x_i}$ that have not already been determined. By hypothesis, $\overline{S}_{x_{i+1}}$ contains enough remaining elements to do this, and we use the smallest among them for the task.
\end{proof}

We demonstrate the proof of Theorem~\ref{thm:main} with examples, recalling Example~\ref{ex: 4,6,8,10,11}(b).

\begin{example}\label{ex:4,6,8,10,11 cont}\
\begin{enumerate}
\item[(a)]Consider $S = \{4,6,8,10,11\}$ and $\mathcal{A} = (10,6,4,11,8)$. We now try to fill between the elements of $\mathcal{A}$, iteratively, to create a permutation with the desired pinnacles in the desired order.
$$\begin{array}{ccccccccccc}
\underline{\hspace{.25in}} &10&\underline{\hspace{.25in}} &6& \underline{\hspace{.25in}} &4& \underline{\hspace{.25in}} &11& \underline{\hspace{.25in}} &8& \underline{\hspace{.25in}}\\
\underline{\hspace{.25in}} &10&\underline{\hspace{.25in}} &6& 1 &4& 2 &11& \underline{\hspace{.25in}} &8& \underline{\hspace{.25in}}\\
\underline{\hspace{.25in}} &10& 3 &6& 1 &4& 2 &11& \underline{\hspace{.25in}} &8& \underline{\hspace{.25in}}\\
\underline{\hspace{.25in}} &10& 3 &6& 1 &4& 2 &11& 5 &8& 7\\
9 &10& 3 &6& 1 &4& 2 &11& 5 &8& 7
\end{array}$$
The fact that this $\mc{A}$ is an admissible ordering is confirmed by computing the $k_x$ values (using the notation of Theorem~\ref{thm:main}) and comparing them to the interruptions in $\mathcal{A}$:
$$k_4 = 1 \hspace{.25in} k_6 = 1 \hspace{.25in} k_8 = 1 \hspace{.25in} k_{10} = 1 \hspace{.25in} k_{11} = 0,$$
while $S_4$ and $S_6$ and $S_{11} = S$ are each uninterrupted, and $S_8$ and $S_{10}$ are each interrupted once.
\item[(b)] Now consider the same set $S$, and the ordering $\mc{A}' = (6,10,4,11,8)$. Since $S_8$ is interrupted twice in $\mathcal{A}'$, and $k_8 = 1$, this $\mc{A}'$ is not an admissible ordering. Indeed, there would be no place for the value $9$ while still creating the desired pinnacles and in the desired order.
\end{enumerate}
\end{example}

Note that the value $k_x = |\overline{S}_x| - |S_x| - 1$ in the statement of Theorem~\ref{thm:main} can also be computed via $x - 2|S_x| - 1$. Similarly, if $S = \{x_1 < x_ 2 < \cdots < x_p\}$, then
\begin{equation}\label{eqn:k_{x_i} in terms of x_i}
k_{x_i} = x_i-2i-1.
\end{equation}

\section{Implications and follow-up questions}\label{sec:follow-up}

One can easily use the main result to confirm the observation of \cite{rusu} and Example~\ref{ex:multilevel bundling}.

\begin{example}
Consider the admissible pinnacle set $S = \{3,5,x\}$ where $x > 5$ (in fact, to be admissible, $x > 6$). Because $|\overline{S}_5| - |S_5| = 1$, we must have that $\{3,5\}$ is uninterrupted in the admissible orderings of $S$.
\end{example}

\begin{example}\label{ex:multilevel bundling redux}
Consider the admissible pinnacle set $S = \{3,5,8,9,13,14\}$. Compute $k_x$ for each $x \in S$:
\begin{align*}
k_3 &= 2-1 - 1 = 0,\\
k_5 &= 3-2 - 1 = 0,\\
k_8 &= 5-3 - 1 = 1,\\
k_9 &= 5-4 - 1 = 0,\\
k_{13} &= 8-5 - 1 = 2, \text{ and}\\
k_{14} &= 8-6 - 1 = 1.
\end{align*}
Therefore, $\{3,5\}$ and $\{3,5,8,9\}$ must each be uninterrupted in the admissible orderings of $S$. We also find that $\{3,5,8\}$ can be interrupted at most once and $\{3,5,8,9,13\}$ can be interrupted at most twice. However, requiring $\{3,5,8,9\}$ to be uninterrupted will guarantee that these requirements for $S_8$ and $S_{13}$ are met.
\end{example}

We can also use Theorem~\ref{thm:main} to more fully explain Example~\ref{ex: 4,6,8,10,11}.

\begin{example}
Let us describe all admissible orderings of the admissible pinnacle set $S = \{4,6,8,10,11\}$. Based on the values $k_x$ computed in Example~\ref{ex:4,6,8,10,11 cont}(a), we find that $\{4\}$, $\{4,6\}$, $\{4,6,8\}$, and $\{4,6,8,10\}$ can each be interrupted at most once, while $\{4,6,8,10,11\}$ must be uninterrupted. Several of these requirements are guaranteed by any permutation with these pinnacles, and the only nontrivial requirement is that $\{4,6,8\}$ be interrupted at most once in any admissible ordering of $S$. With this information, we can count admissible orderings of $S$. Namely, there are $5!$ permutations of $S$, and we disallow those in which $\{4,6,8\}$ occur in alternating positions: $5! - 3!2! = 108$.
\end{example}

The strictest requirements we find from Theorem~\ref{thm:main} are when $|S_x| = (x-1)/2$, in which case $S_x$ must be uninterrupted. Because $S$ is a set of integers, we can draw the following conclusion.

\begin{corollary}\label{cor:only check odd}
The only sets $S_x$ which require no interruption are when $x \in S$ is odd.
\end{corollary}

An obvious open problem is to find a function for computing the number of admissible orderings of a given pinnacle set.

In another direction, we can ask a question analogous to the one that led to the recent work of \cite{diaz-lopez harris huang insko nilsen,rusu}. Namely, is there a class of operations that one may apply to any permutation $w$ with $\Pin(w) = S$ and pinnacles appearing in a fixed order $\mc{A}$, to obtain any other $w'$ with the same pinnacle set and ordering of the pinnacles, and no other permutations?

Similarly, is there a way to use properties of $S$ (without having to look at all permutations with pinnacle set $S$) to endow its orderings with a partial order such that $\mc{A} \le \mc{A}'$ in the partial order, then the number of permutations with pinnacles appearing in the order $\mc{A}$ is less than or equal to the number with pinnacles appearing in the order $\mc{A}'$?

Given Theorem~\ref{thm:main} and Corollary~\ref{cor:only check odd}, it is natural to wonder about pinnacle sets with ``extreme'' amounts of uninterruption. For example, what does ``extreme'' mean in this context? What pinnacle sets have it? How many are there? What does this imply for the permutations with those pinnacle sets? How many are there? Etc. One could ask similar questions about pinnacle sets with ``average'' amounts of interruption.

We conclude by mentioning one version of extremism here; namely, those pinnacle sets in which every ordering is admissible.

\begin{definition}
An admissible pinnacle set $S$ is \emph{maximally admissible} if every ordering of $S$ is admissible.
\end{definition}

A pinnacle set $S$ that is maximally admissible has $|S|!$ admissible pinnacle orderings, and fewer otherwise.

\begin{corollary}\label{cor:maximally-admissible}
An admissible pinnacle set $S = \{x_1 < x_2 < \cdots < x_p\}$ is maximally admissible if and only if 
for each $1<i < p$ the condition $x_i\geq \min\{p+i+1,3i\}$ holds.
 \end{corollary}

\begin{proof}
For a fixed $i$, each interruption of $S_{x_i}$ is due to an element of $S\setminus S_{x_i}$, so the number of interruptions of $S_{x_i}$
is bounded from above by $|S \setminus S_{x_i}|=p-i$. Every ordering is admissible if and only if $S_{x_i}$ may be interrupted at least 
$\min\{|S_{x_i}| - 1,|S \setminus S_{x_i}|\}=\min\{i-1,p-i\}$ times, for each $1<i<p$. 
(The set $S$ itself cannot be interrupted, nor can $S_{x_1}$.) 
According to Theorem~\ref{thm:main} and equation (1), that means $\min\{i-1,p-i\}\leq k_{x_i}=x_i-2i-1$ and the conclusion follows.
\end{proof}

Example~\ref{ex:multilevel bundling redux} shows that in Theorem~\ref{thm:main}, some of the tests concerning the number of interruptions
of $S_x$ may be unnecessary. We identify below a subset of tests avoiding a lot of these redundancies.

\begin{corollary}\label{cor:eliminating redundancy}
Let $S=\{x_1 < x_2 < \cdots < x_p\}$ be an admissible pinnacle set. An ordering $\mc{A}$ of $S$ is admissible if and only 
if for each $1<i<p$ such that:

\begin{itemize}
\item[\emph{($i$)}] $x_i<\min\{p+i+1,3i\}$, and
\item[\emph{($ii$)}] $x_{i-1}+2\geq x_i\leq x_{i+1}-2$.
\end{itemize}

\noindent the set $S_{x_i}$ is interrupted at most $k_{x_i}=|\overline{S}_{x_i}| - |S_{x_i}|-1$ times in $\mc{A}$.

\end{corollary}

\begin{proof} The forward direction is directly implied by Theorem~\ref{thm:main}.

In order to prove the reverse direction, we will show that even for the pinnacles $x_j$ that do not satisfy both conditions~($i$) and ($ii$),  
the set $S_{x_j}$ is interrupted at most $k_{x_j}$ times. Note that neither $S$ nor $S_{x_1}$ can be interrupted, and so any bound on the number of interruptions permitted for these sets is trivially satisfied. 

Suppose that $x_j\in S$ is a pinnacle that does not satisfy both conditions~($i$) and ($ii$), with $1 < j < p$.  If condition~($i$) is false for $x_j$, then $x_j\geq \min\{p+j+1,3j\}$. Corollary \ref{cor:maximally-admissible} implies that
$k_{x_j}$ is large enough to allow as many interruptions of $S_{x_j}$ as possible, and the conclusion follows.

We assume now that $x_j$ does not satisfy condition~($ii$). By equation~\eqref{eqn:k_{x_i} in terms of x_i}, we have that $x_{l-1}+2\geq x_l$ if and only if $k_{x_{l-1}}\geq k_{x_l}$, for all $1 < l \le p$. Thus
$$x_{j-1}+2\geq x_j\leq x_{j+1}-2\,\, \hbox{if and only if}\,\, k_{x_{j-1}}\geq k_{x_{j}}\leq k_{x_{j+1}}$$
For each $1<l<p$, if $h$ is the number of interruptions of $S_{x_l}$ in $\mc{A}$, then
\begin{itemize}
\item[(a)] the number of interruptions of $S_{x_{l+1}}$ in $\mc{A}$ is at most $h+1$, and
\item[(b)] the number of interruptions of $S_{x_{l-1}}$ in $\mc{A}$ is at most $h+1$.
 \end{itemize}
Then there are two possible cases.

\begin{quotation}
\noindent Case 1: $k_{x_{j-1}}< k_{x_j}$.

Let $h \leq j-1$ be maximal such that $x_{h}$ satisfies conditions~($i$) and ($ii$), or $h=1$. Thus $h=1$, or 
$k_{x_{h-1}}\geq k_{x_{h}} <k_{x_{h+1}}< \cdots <k_{x_{j-1}}<k_{x_j}$. By hypothesis, the set $S_{x_{h}}$ is interrupted at most $k_{x_{h}}$ times. Furthermore, by affirmation~(a) applied $j-h$ times, 
we deduce that the number of interruptions of $S_{x_j}$ is at most $k_{x_{h}}+(j-h)$. The inequalities between $k_{x_h}$ and $k_{x_j}$ imply that $k_{x_{h}}+(j-h)\leq k_{x_j}$, as needed.

\vspace{.1in}

\noindent Case 2: $k_{x_{j}}> k_{x_{j+1}}$.

Let $h \geq j+1$ be minimal such that $x_{h}$ satisfies conditions~($i$) and ($ii$), or $h=p$. Thus $h=p$, or 
$k_{x_{h+1}} \ge k_{x_{h}}<k_{x_{h-1}}<\cdots<k_{x_{j+1}}<k_{x_j}$. By hypothesis, the set $S_{x_{h}}$ is interrupted at most $k_{x_{h}}$ times. Furthermore, by affirmation~(b) applied $h-j$ times, 
we have that the number of interruptions of $S_{x_j}$ is at most $k_{x_{h}}+(h-j)$. Again, the inequalities between $k_{x_h}$ and $k_{x_j}$ imply that $k_{x_{h}}+(h-j)\leq k_{x_j}$, as needed.

\end{quotation}

\noindent By Theorem~\ref{thm:main}, we deduce that $\mc{A}$ is admissible.
  
\end{proof}

To demonstrate Corollary~\ref{cor:eliminating redundancy}, we return to Example~\ref{ex:multilevel bundling redux}.

\begin{example}
Consider the admissible pinnacle set $S = \{3,5,8,9,13,14\}$. Following Corollary~\ref{cor:eliminating redundancy}, the only interruptions we need to check are for the pinnacles $5, 9 \in S$. We have $k_5 = 0$, meaning that $S_5$ must be uninterrupted, and $k_9 = 0$, meaning that $S_9$ must be uninterrupted. This latter fact forces $S_8$ to be interrupted at most once (necessarily by $9$), while $S_{13}$ can never be interrupted more than once (because the only remaining pinnacle is $14$). The sets $S_3$ and $S_{14} = S$ are necessarily uninterrupted. Thus we have recovered the conclusions of Example~\ref{ex:multilevel bundling redux}.
\end{example}

\section*{Acknowledgements}

We are grateful for the thoughtful feedback of Alexander Diaz-Lopez and Erik Insko. We also appreciate the careful reading of anonymous referees.

\section{Declarations}

\paragraph{Funding.} B.E. Tenner - Research partially supported by Simons Foundation Collaboration Grant for Mathematicians 277603 and by a University Research Council Competitive Research Leave from DePaul University.

\paragraph{Conflicts of interest/Competing interests.} 
The authors declare that they have no conflict of interest.

\paragraph{Availability of data and material.} Not applicable.

\paragraph{Code availability.} Not applicable.

\end{document}